\documentclass[12pt,a4paper]{amsart} 

\usepackage{latexsym}

\usepackage[a4paper , lmargin = {2.5cm} , rmargin = {2.5cm} , tmargin = {2.5cm} , bmargin = {2.5cm} ]{geometry} 

\usepackage{geometry}
\usepackage{pinlabel}         
\usepackage{graphicx}
\usepackage{amssymb, amsthm}
\usepackage{epstopdf}
\usepackage{romannum}
\usepackage{tikz}
\usepackage{subfig}
\usepackage{MnSymbol}
\usepackage[arrow, matrix, curve]{xy}

	\def\CAT{{{\rm CAT}$(0)$}}

\theoremstyle{plain}
\newtheorem*{NewTheoremA}{Theorem A}	
\newtheorem*{NewCorollaryB}{Corollary B}
\newtheorem*{NewCorollaryC}{Corollary C}

\newtheorem*{Question}{Question}

\newtheorem*{Conjecture}{Conjecture}

\newtheorem{theorem}{Theorem}[section]

\newtheorem{proposition}[theorem]{Proposition}

\title[]{A note on locally elliptic actions on cube complexes}
\author{Nils Leder and Olga Varghese}
\date{\today}

\address{Nils Leder\\
Department of Mathematics\\
M\"unster University\\ 
Einsteinstra\ss e 62\\
48149 M\"unster (Germany)}
\email{n\_lede02@uni-muenster.de}

\address{Olga Varghese\\
Department of Mathematics\\
M\"unster University\\ 
Einsteinstra\ss e 62\\
48149 M\"unster (Germany)}
\email{olga.varghese@uni-muenster.de}

\begin{document}
\pagenumbering{arabic}
\begin{abstract}
We deduce from Sageev's results that whenever a group acts locally elliptically on a finite dimensional \CAT\ cube complex, then it must fix a point. 
As an application, we give an example of a group $G$ such that $G$ does not have property (T), but $G$ and all its finitely generated subgroups can not act without a fixed point on a finite dimensional \CAT\ cube complex, answering a question by Barnhill and Chatterji. 
\end{abstract}
\maketitle

\section{Introduction}
The questions we shall investigate in this note are concerned with fixed points on \CAT\ cube complexes. Roughly speaking, a cube complex is a union of cubes of any dimension which are glued together along isometric faces. Let $\mathcal{C}$ be a class of finite dimensional \CAT\ cube complexes. A group $G$ is said to have property ${\rm F}\mathcal{C}$ if any simplicial
action of $G$ on any member of $\mathcal{C}$ has a fixed point. For a subclass $\mathcal{A}$ consisting of simplicial trees the study of property ${\rm F}\mathcal{A}$ was initiated by Serre \cite{Serre}. 

Bass introduced a weaker property ${\rm F}\mathcal{A'}$ for groups in \cite{Bass}. A group has property ${\rm F}\mathcal{A'}$ if any simplicial action of $G$ on any member of $\mathcal{A}$ is locally elliptic, i.e. each $g\in G$ fixes some point on a tree. We define a generalization of property ${\rm F}\mathcal{A'}$. A group $G$ has property ${\rm F}\mathcal{C'}$ if any simplicial action of $G$ on any member of $\mathcal{C}$ is locally elliptic, i.e. each $g\in G$ fixes some point on a \CAT\ cube complex. 

It is well known that a finitely generated group which is acting locally elliptically on a simplicial tree has a global fixed point, see \cite[\S 6.5 Corollary 2]{Serre}. The following result is due to Sageev and follows from the proof of Theorem 5.1 in \cite{Sageev}.
\begin{NewTheoremA}
Let $G$ be a finitely generated group acting by simplicial isometries on 
a finite dimensional \CAT\ cube complex. If the $G$-action is locally elliptic, then $G$ has a global fixed point. 

In particular, a finitely generated group $G$ has property ${\rm F}\mathcal{C'}$ iff $G$ has property ${\rm F}\mathcal{C}$.
\end{NewTheoremA}
Before we state the corollaries of Theorem A, we observe that the result in Theorem A is not true for infinite dimensional \CAT\ cube complexes. Let $G$ be a finitely generated torsion group. Then, by Bruhat-Tits fixed point theorem \cite[Corollary II 2.8]{Haefliger} follows, that $G$ has property ${\rm F}\mathcal{C'}$ and thus by Theorem A the group $G$ has property ${\rm F}\mathcal{C}$. Free Burnside groups are finitely generated torsion groups and thus these groups have always property ${\rm F}\mathcal{C}$, but many of these groups act without a fixed point on infinite dimensional \CAT\ cube complexes, see \cite[Theorem 1]{Osajda}.

The next corollary follows from Theorem A and is known in a case of trees by a result of Tits \cite[Proposition 3.4]{Tits}.
\begin{NewCorollaryB}
Let $G$ be a group acting by simplicial isometries on a finite dimensional \CAT\ cube complex $X$. If the $G$-action is locally elliptic, then $G$ has a global fixed point in $X\cup\partial X$, where $\partial X$ denotes the visual boundary of $X$.
\end{NewCorollaryB}
\begin{proof}
For the proof we need the following result by Caprace \cite[Theorem 1.1]{Caprace}:

Let $X$ be a finite dimensional \CAT\ cube complex and $\left\{X_\alpha\right\}_{\alpha\in A}$ be a filtering family of closed convex non-empty subsets. Then either the intersection $\bigcap\limits_{\alpha\in A} X_\alpha$ is non-empty or the intersection of the visual boundaries $\bigcap\limits_{\alpha\in A}\partial X_\alpha$ is a non-empty subset of $\partial X$.

Recall that a family $\mathcal{F}$ of subsets of a given set is called \emph{filtering} if for all $E, F\in\mathcal{F}$ there exists $D\in\mathcal{F}$ such that $D\subseteq E\cap F$.

Let $X$ be a finite dimensional \CAT\ cube complex and $\Phi$ a simplicial action of $G$ on $X$. For $S\subseteq G$ we define the following set ${\rm Fix}(S)=\left\{x\in X\mid \Phi(s)(x)=x\text{ for all } s\in S\right\}$. The set ${\rm Fix}(S)$ is closed and convex. If $S$ is a finite set, then it follows by Theorem A that ${\rm Fix}(S)$ is non-empty. Further, we define ${\rm Fix}(G)^\partial=\left\{\xi\in\partial X\mid \Phi(g)(\xi)=\xi\text{ for all }g\in G\right\}.$

Now we consider the following family
$\mathcal{F}=\left\{{\rm Fix}(S)\mid S\subseteq G\text{ and }\#S<\infty\right\}.$ If $S,T \subseteq G$ are finite subsets, we have ${\rm Fix}(S \cup T) \subseteq {\rm Fix}(S) \cap {\rm Fix}(T)$ and thus $\mathcal{F}$ is a filtering. 
The result of Caprace stated above implies that 
\[
\bigcap\mathcal{F}={\rm Fix}(G)\text{ is non-empty }
\]
or 
\[
\bigcap\left\{\partial{\rm Fix}(S)\mid S\subseteq G\text{ and } \#S<\infty\right\}\subseteq{\rm Fix}(G)^\partial \text{ is non-empty.}
\] 
\end{proof}

Since the Davis realization of a right-angled building carries the structure of a finite-dimensional $\mathrm{CAT}(0)$ cube complex, we can apply Corollary B to confirm the following conjecture by Marquis \cite[Conjecture 2]{Marquis}) in the special case of right-angled buildings.

\begin{Conjecture}
Let $G$ be a group acting by type-preserving simplicial isometries on a building $\Delta$. If the $G$-action on the Davis realization $X$ of $\Delta$ is locally elliptic, then $G$ has a global fixed point in $X \cup \partial X$.
\end{Conjecture}

Another fixed point property of interest is  Kazhdan's property (T). Niblo and Reeves proved in \cite[Theorem B]{Niblo} that if a group $G$ has Kazhdan's property (T), then $G$ also has property ${\rm F}\mathcal{C}$. Barnhill and Chatterji raised the following question  \cite[Question 5.3 ]{Barnhill}.

\begin{Question}
Is ${\rm F}\mathcal{C}$ equivalent to (T), or does there exist a group $G$ such that $G$ does not have property (T), but $G$ and all its finite-index subgroups have property ${\rm F}\mathcal{C}$? 
\end{Question}

With the next result we can answer this question in the negative.
\begin{NewCorollaryC}
Let $G$ be the first Grigorchuk group. Then $G$ and all its finitely generated subgroups have property ${\rm F}\mathcal{C}$, but $G$ doesn't have property (T). In particular, all finite-index subgroups of $G$ also have property ${\rm F}\mathcal{C}$.
\end{NewCorollaryC}
\begin{proof}
The first Grigorchuk group $G$ is a finitely generated infinite torsion group (see \cite{Grigorchuk}) and thus $G$ and all its finitely generated subgroups have property 
${\rm F}\mathcal{C}$. But $G$ does not have property (T) since $G$ is amenable, see \cite{Grigorchuk2}.
\end{proof}
Further, many free Burnside groups have property ${\rm F}\mathcal{C}$, but don't have property (T), see \cite[Theorem 1]{Osajda}. Other examples of groups with property ${\rm F}\mathcal{C}$ and without property (T) were given by Cornulier in \cite{Cornulier}.

\subsection*{Acknowledgement}
We would like to thank R\'emi Coulon for poiting us on Theorem 5.1 in \cite{Sageev}.

\section{Proof of Theorem A}
In this section we give a proof of Theorem A which is hidden in the proof of Theorem 5.1 in \cite{Sageev} by Sageev. For definitions and properties of \CAT\ cube complexes see \cite{Sageev}.

We first need the following result.
\begin{proposition}
\label{ThreeHyperplanes}
Let $X$ be a $d$-dimensional \CAT\ cube complex and $\mathcal S$ be a finite set of hyperplanes in $X$. If $\#\mathcal S\geq d+d\cdot(d+1)$, then there exist three hyperplanes in $\mathcal S$ that do not intersect pairwise.
\end{proposition}
\begin{proof}
Let $\mathcal{T}=\left\{J_1, \ldots, J_k\right\}\subseteq S$ be a maximal set of pairwise intersecting hyperplanes. Then by Helly's Theorem for \CAT\ cube complexes or \cite[Theorem 4.14]{Sageev} follows that $\bigcap\mathcal{T}$ is not empty. Further, since the dimension of $X$ is $d$ we have:  $k\leq d$. 
By maximality of $\mathcal{T}$, for each hyperplane $J \in\mathcal{S}-\mathcal{T}$ there exists $i=1, \ldots ,k$ such that $J \cap J_i = \emptyset$. This yields a well-defined map
$$q:\mathcal{S}-\mathcal{T} \rightarrow \{1, \ldots, k\}, J \mapsto \min\{i \mid J \cap J_i = \emptyset\}.$$
Let $B_i$ denote the preimage $q^{-1}(i)$ for $i=1, \ldots, n$. Since $\#\mathcal{S}\geq d+d\cdot(d+1)$ and $k\leq d$, we have $\#(\mathcal{S}-\mathcal{T}) \geq d \cdot (d+1)$. Thus, by the pigeon-hole principle there exists $j \in \{1, \ldots, k\}$ such that $\#B_j \geq d+1$. By maximality of $\mathcal{T}$, not all hyperplanes of $B_j$ intersect pairwise, i.e there are $H_1, H_2\in B_j $ such that $H_1\cap H_2=\emptyset$. Then, $J_j, H_1, H_2$ are three hyperplanes that do not intersect each other. 
\end{proof}

Now we are ready to prove Theorem A.
\begin{proof}[Proof of Theorem A]
Let $G$ be a finitely generated group with a symmetric generating set $Y=\left\{ g_1, \ldots, g_n\right\}$. Let $X$ be a  $d$-dimensional \CAT\ cube complex, $v\in X$ be a vertex and $G\rightarrow{\rm Isom}(X)$
be a simplicial locally elliptic action.

For $i=1,\ldots, n$ we choose a combinatorial geodesic $\lambda_i$ from $v$ to $g_{i}(v)$ . Further, we denote by $\mathcal{S}_i$  the set of hyperplanes crossed by $\lambda_i$. We have $\#\mathcal{S}_{i}=D(v, g_{i}(v))$, where we denote by $D$ the metric on the $1$-skeleton of $X$. Hence the union $\mathcal{S}:=\bigcup\limits_{i=1}^n\mathcal{S}_i$ is a finite set.

Let us assume that the action has no global fixed point. Then by the Bruhat-Tits fixed point Theorem  follows that the orbit of $v$ is unbounded. Thus, there exists $g\in G$ such that
\[
N:=D(v, g(v))\geq \#\mathcal{S}\cdot (d+d(d+1)).
\] 
Since $Y$ generates $G$, we can write $g=g_{i_1}\ldots g_{i_l}$ with $g_{i_j}\in Y$ for $i=1,\ldots, l$. We define 
\[
v_j:=g_{i_1}\ldots g_{i_j}(v) \text{ and } \gamma_j:=g_{i_1}\ldots g_{i_j}(\lambda_{i_{j+1}}).
\]
The map $\gamma_j$ is a combinatorial geodesic from $v_j$ to $v_{j+1}$. Hence $\alpha:=\gamma_l\ldots\gamma_1\lambda_{g_{i_1}}$ 
is a combinatorial path from $v$ to $g(v)$. Since $D(v, g(v))=N$, there exists a set of hyperplanes $\mathcal{T}=\left\{K_1,\ldots, K_N\right\}$ such that $\alpha$ crosses each hyperplane in $T$. 

By construction, for each $K_i$ in $\mathcal{T}$ there exists $J\in\mathcal{S}$ such that $K_i=hJ$ for some $h\in G$. By pigeon-hole principle  there exists a hyperplane $J\in\mathcal{S}$ such that
\[
\#\left\{K\in\mathcal{T}\mid \exists h \in G: K=hJ\right\}\geq d+d(d+1).
\]
By Proposition \ref{ThreeHyperplanes} there exist three hyperplanes $h_1J, h_2J$ and $h_3J$ in $\left\{K\in\mathcal{T}\mid \exists h \in G: K=hJ\right\}$ %\text{ for some }h\in G\right\} 
whose pairwise intersection is empty. But each of these hyperplanes is crossed precisely once by a combinatorial geodesic from $v$ to $g(v)$. Therefore one of these hyperplanes separates the other two. 

It is not  difficult to verify the following: If there exist a hyperplane $J\subseteq X$ and $g, h\in G$ such that $J, gJ, hJ$ do not intersect pairwise and $gJ$ separates $J$ and $hJ$, then $g, h$ or $hg^{-1}$ is hyperbolic.

This completes the proof.
\end{proof}

\end{document}